\documentclass[a4paper,9pt]{amsart}
\usepackage{blindtext}
\usepackage{graphicx}

\usepackage[utf8]{inputenc}
\usepackage{amssymb,amsfonts,amsmath}
\usepackage{enumerate}
\usepackage{mathrsfs}
\usepackage{eucal}
\usepackage[letterpaper, portrait, margin=0.8in]{geometry}
\usepackage{amsthm}
\usepackage[english]{babel}
\usepackage{comment}
\usepackage[all,cmtip,2cell]{xy}
\usepackage{lipsum}
\usepackage[makeroom]{cancel}
\usepackage{lmodern}
\newcommand{\pushoutcorner}[1][dr]{\save*!/#1-1.2pc/#1:(-1,1)@^{|-}\restore}
\newcommand\blfootnote[1]{%
  \begingroup
  \renewcommand\thefootnote{}\footnote{#1}%
  \addtocounter{footnote}{-1}%
  \endgroup
}

\UseTwocells

\makeatletter
\newcommand{\colim@}[2]{%
  \vtop{\m@th\ialign{##\cr
    \hfil$#1\operator@font colim$\hfil\cr
    \noalign{\nointerlineskip\kern1.5\ex@}#2\cr
    \noalign{\nointerlineskip\kern-\ex@}\cr}}%
}
\newcommand{\colim}{%
  \mathop{\mathpalette\colim@{\rightarrowfill@\scriptscriptstyle}}\nmlimits@
}
\renewcommand{\varprojlim}{%
  \mathop{\mathpalette\varlim@{\leftarrowfill@\scriptscriptstyle}}\nmlimits@
}
\renewcommand{\varinjlim}{%
  \mathop{\mathpalette\varlim@{\rightarrowfill@\scriptscriptstyle}}\nmlimits@
}
\makeatother

\begin{document}

\theoremstyle{plain}

\newtheorem{thm}{Theorem}[section]
\newtheorem{prop}[thm]{Proposition}
\newtheorem{lem}[thm]{Lemma}

\theoremstyle{definition}

\newtheorem{cor}[thm]{Corollary}

\newtheorem{term}[thm]{Terminology}
\newtheorem{defn}[thm]{Definition}
\newtheorem{defns}[thm]{Definitions}
\newtheorem{con}[thm]{Construction}
\newtheorem{exmp}[thm]{Example}
\newtheorem{nexmp}[thm]{Non-Example}
\newtheorem{exmps}[thm]{Examples}
\newtheorem{notn}[thm]{Notation}
\newtheorem{notns}[thm]{Notations}
\newtheorem{addm}[thm]{Addendum}
\newtheorem{exer}[thm]{Exercise}
\newtheorem{obs}[thm]{Observation}

\theoremstyle{remark}
\newtheorem{rem}[thm]{Remark}

\newtheorem{rems}[thm]{Remarks}
\newtheorem{warn}[thm]{Warning}
\newtheorem{sch}[thm]{Scholium}

\title{The fundamental $\infty$-groupoid of a parametrized family}
\author{Karthik Yegnesh}

%\author{\hspace{75mm}Karthik Yegnesh^{1}}\footnote{Methacton High School\\
%\hspace{10mm}1005 Kriebel Mill Rd, Eagleville, PA 19403\\
%karthik.yegnesh@gmail.com}

\vspace{-20mm}

\maketitle
\begin{abstract}
Given an $\infty$-category $\mathcal{C}$, one can naturally construct an $\infty$-category $\mathrm{Fam}(\mathcal{C})$ of families of objects in $\mathcal{C}$ indexed by $\infty$-groupoids. An ordinary categorical version of this construction was used by Borceux and Janelidze in the study of generalized covering maps in categorical Galois theory. In this paper, we develop the homotopy theory of such ``parametrized families" as generalization of the classical homotopy theory of spaces. In particular, we study homotopy-theoretical constructions that arise from the \textit{fundamental $\infty$-groupoids} of families in an $\infty$-category. In the same spirit, we show that $\mathrm{Fam}(\mathcal{C})$ admits a Grothendieck topology  which generalizes Carchedi's canonical/epimorphism topology on certain $\infty$-topoi.
\end{abstract}
\section{Introduction}
\subsection{Motivation}
In their generalization of classical Galois theory in [2], Borceux and Janelidze draw connections between the \textit{coproduct completion} of categories and the theory of locally connected topological spaces. The coproduct completion $\mathrm{Fam}(\mathcal{C})$ of a category $\mathcal{C}$ is the category whose objects are families of objects in $\mathcal{C}$ parameterized by sets and whose morphisms are maps between members of the families in question induced by functions on the indexing sets. As described in [2], every category of the form $\mathrm{Fam}(\mathcal{C})$ is equipped with a ``family fibration" functor $\pi_{0}:\mathrm{Fam}(\mathcal{C})\rightarrow\mathrm{Set}$, which sends each family to its indexing set. Geometrically, $\pi_{0}$ is a generalization of the usual connected components functor for topological spaces, since a family of objects in $\mathcal{C}$ can be viewed as a ``disjoint union" of its members - the members being the ``connected components" of some sort of generalized space. When $\mathcal{C}$ is instead a $(2,1)$-category, its $\mathrm{Fam}(\mathcal{C})$ is the same as the 1-categorical case except that the families are parametrized by groupoids instead of sets (and is a colimit completion with respect to groupoid-indexed diagrams). Thus, the family fibration functor is $\mathrm{Grpd}$-valued and can be viewed as an analog of the fundamental groupoid of a topological space. This trend continues for $(n,1)$-categories as $n\rightarrow\infty$, so it is natural to expect that one can define the ``fundamental $\infty$-groupoid" of parametrized families in an $\infty$-category. This would provide an $\infty$-categorical generalization of the analogy between families of objects in categories and locally connected spaces studied by Borceux and Janelidze - and thus lead to a non-trivial homotopy theory of parametrized families in $\infty$-categories.\\

  %Specifically, we study topological invariants (e.g the fundamental $\infty$-groupoid and the based fundamental group) of such families which derive from the homotopy theory of spaces. Our constructions recover classical invariants of topological spaces when considering the $\infty$-category $\mathrm{Fam}(*)\simeq\mathrm{Grpd}_{\infty}$ of spaces, where $*$ is the trivial $\infty$-groupoid. We also define a Grothendieck topology on the $\infty$-category of parametrized families in an $\infty$-category which generalizes the ``canonical" topology on the $\infty$-topos of $\infty$-groupoids of Carchedi [3]. In the Appendix, we describe how our constructions relate to Joyal's notion of a locus and state some of our tangential results. 

    \subsection{ Goal and Outline}The goal of this paper is to develop the homotopy theory of families in $\infty$-categories as an extension of the homotopy theory of topological spaces. In \S 2, we will recall some background information about extensive categories and Grothendieck topologies on an $\infty$-category. In \S 3, we will introduce the $\infty$-category $\mathrm{Fam}(\mathcal{C})$ of parametrized families in an $\infty$-category and prove general categorical results about $\mathrm{Fam}(\mathcal{C})$ that we need in subsequent sections. In \S 4, we define and study the fundamental $\infty$-groupoids and fundamental groups of objects in $\mathrm{Fam}(\mathcal{C})$. In \S 5, we use the construction of \S 4 to construct a Grothendieck topology on $\mathrm{Fam}(\mathcal{C})$ which generalizes Carchedi's ``canonical" topology on the $\infty$-topos of $\infty$-groupoids. In \S 6, we describe Joyal's notion of a(n) ($\infty$-)locus and how our results relate to it. We also state some tangential results of ours there. In the final section, we describe some avenues for future investigation.
  
\subsection{Conventions}
We will assume an understanding of basic topos theory and higher category theory. By an $n$-category, we mean an $(n,1)$-category. In particular, an $\infty$-category is a category enriched over $\infty$-groupoids/Kan complexes, i.e an $(\infty,1)$-category. $\mathrm{Cat}_{\infty}$ (resp. $\mathrm{Grpd}_{\infty}$) denotes the $(\infty,2)$-category of $\infty$-categories (resp. $\infty$-category of $\infty$-groupoids). A topos always means a Grothendieck/sheaf topos. \blfootnote{Methacton High School\\
1005 Kriebel Mill Rd, Eagleville, PA 19403\\
karthik.yegnesh@gmail.com}

\section{Background}

\subsection{Extensive $\infty$-categories}
We define extensive $\infty$-categories as a slightly weakened version of Barwick's disjunctive $\infty$-categories defined in Section 4 of [1]. They are identical except that we don't require them to be closed under finite limits.

\begin{defn}
Let $\mathcal{C}$ be an $\infty$-category. $\mathcal{C}$ is \textit{extensive} if, for an arbitrary collection $\{X_{i}\}_{i\in I}$ of objects in $\mathcal{C}$, the canonical coproduct functor $\coprod:\prod_{i\in I}\mathcal{C}_{/X_{i}}\rightarrow\mathcal{C}_{/\coprod_{i\in I}X_{i}}$ is a categorical equivalence. 
\end{defn}

\begin{exmp}
Any $\infty$-topos (e.g $\mathrm{Grpd}_{\infty}$, $[X^{op},\mathrm{Grpd}_{\infty}]$ for a small category $X$, $\mathcal{S}\mathrm{hv}_{\infty}(S)$ for an $\infty$-site $S$, etc.) is extensive.
\end{exmp}

\begin{defn}
Let $\mathcal{C}$ be an $\infty$-category. An object $X\in \mathcal{C}$ is \textit{connected} if $\mathrm{Hom}_{\mathcal{C}}(X,-):\mathcal{C}\rightarrow\mathrm{Grpd}_{\infty}$ preserves all coproducts.
\end{defn}
\begin{rem}
There is a simpler description of connected objects in extensive $\infty$-categories. Namely, if $\mathcal{C}$ is an extensive $\infty$-category, then an object $X$ is connected if and only if for any coproduct decomposition $X=X_{1}\coprod X_{2}$, exactly one of the $X_{i}$ is not initial.
\end{rem}
\begin{exmp}
A topological space/$\infty$-groupoid is connected (categorically) precisely if it connected in the usual sense. In the 1-truncated case, an object in $\mathrm{Set}$ is connected if and only if it is a singleton. Additionally, a scheme is a connected object in the category of schemes if and only if it is a connected scheme, i.e its underlying topological space is connected.
\end{exmp}

\subsection{Grothendieck Topologies} Let $\mathcal{C}$ be an $\infty$-category. A Grothendieck topology on $\mathcal{C}$ allows us to treat objects of $\mathcal{C}$ like open sets of a topological space. In this subsection, we will briefly review some key ideas relevant to Grothendieck topologies. 

\begin{defn}
A \textit{covering} of an object $X\in\mathcal{C}$ is a set of maps $\{f_{i}:X_{i}\rightarrow X\}_{i\in I}$ with a common codomain $X$ that satisfy the following conditions:
\end{defn}
\begin{itemize}
\item If $X'\xrightarrow{\simeq} X$ is an equivalence, then the singleton set $\{X'\xrightarrow{\simeq} X\}$ is a covering.
\item If $\{f_{i}:X_{i}\rightarrow X\}_{i\in I}$ is a covering and $g:Y\rightarrow X$ is a map in $\mathcal{C}$, then the pullbacks $X_{i}\times_{X}Y\rightarrow Y$ exist for each $i\in I$ and $\{X_{i}\times_{X}Y\rightarrow Y\}_{i\in I}$ is a covering of $Y$.
\item If $\{f_{i}:X_{i}\rightarrow X\}_{i\in I}$ is a covering and each $X_{i}$ is equipped with a covering $\{f_{ij}:X_{ij}\rightarrow X_{i}\}_{j\in J}$, then the composite family $\{f_{i}\circ f_{ij}:X_{ij}\rightarrow X\}_{ij}$ covers $X$.
\end{itemize}

\begin{defn}
A \textit{Grothendieck topology} $\tau$ on an $\infty$-category $\mathcal{C}$ is an assignment of coverings $\{f_{i}:X_{i}\rightarrow X\}_{i\in I}$ to each object $X\in\mathcal{C}$. An $\infty$-category equipped with a Grothendieck topology is a (Grothendieck) $\infty$-site. We will suppress the ``$\infty-$" when it is clear that we are in the $\infty$-categorical context. If $\tau$ is a topology on $\mathcal{C}$, then we denote the associated site by $(\mathcal{C},\tau)$ unless the context is clear.
\end{defn}
\begin{rem}
Note that the above definition is often called a Grothendieck \textit{pretopology} - this acts as a ``basis" for a Grothendieck topology on an $\infty$-category.
\end{rem}
\begin{exmp}
Let $\mathrm{CartSp}$ denote the category of smooth manifolds of the form $\mathbb{R}^{n}$ for $n\in\mathbb{N}$ and smooth functions. There is a natural topology on $\mathrm{CartSp}$ whose covering families are the usual open covers.
\end{exmp}
\begin{exmp}\label{effe}
Let $\mathbb{H}$ be a 1-topos. There is a Grothendieck topology, namely the \textit{canonical topology}, on $\mathbb{H}$ whose covering families are families $\{f_{i}:X_{i}\rightarrow X\}_{i\in I}$ that are jointly epimorphic (the induced map $\coprod_{i\in I}X_{i}\rightarrow X$ is an epimorphism). There is an $\infty$-toposic refinement of this notion, described in [3, Definition 2.2.5]. Namely, there is a Grothendieck topology on any $\infty$-topos $\mathbb{H}$ whose covering families are (generated by) sets of maps $\{X_{i}\rightarrow X\}_{i}$ such that the induced map $\coprod_{i\in I}X_{i}\rightarrow X$ is an \textit{effective epimorphism} (see Definitions \ref{l} and \ref{h} and Example \ref{luu}). This is known as the \textit{epimorphism} topology. 
\end{exmp}
\begin{defn}\label{l}
Let $\mathcal{C}$ be an $\infty$-category with pullbacks. The \textit{\v{C}ech nerve} of a map $f:X'\rightarrow X$ is the simplicial object $\check{C}(f)_{\bullet}:\Delta^{op}\rightarrow\mathcal{C}$ sending $[k]$ to the $k$-fold fiber product $X'\times_{X}X'\times_{X}\ldots\times_{X}X'$ of $X'$ with itself.
\end{defn}
\begin{defn}\label{h}
Let $f:X\rightarrow Y$ be a map in an $\infty$-category $\mathcal{C}$ such that $\check{C}(f)_{\bullet}$ exists ($\mathcal{C}$ being closed under pullbacks ensures this). Let $\Delta_{\alpha}$ denote the augmented simplex category. We can construct an augmented simplicial object $\check{C}(f)_{\bullet}^{\sharp}:\Delta_{\alpha}^{op}\rightarrow\mathcal{C}$ out of $\check{C}(f)_{\bullet}$ by attaching $Y$ to $\check{C}(f)_{\bullet}$ via $f$, i.e by setting $d^{-1}=f$ and $\check{C}(f)_{[-1]}'=Y$. $f$ is an \textit{effective epimorphism} if $\check{C}(f)_{\bullet}^{\sharp}$ exhibits $Y$ as the colimit of $\check{C}(f)_{\bullet}$.
\end{defn}
\begin{exmp}[6, Corollary 7.2.1.15]\label{luu}
Let $f:X\rightarrow Y$ be a map of $\infty$-groupoids. $f$ is an effective epimorphism precisely if the induced function $\pi_{0}(f):\pi_{0}(X)\rightarrow\pi_{0}(Y)$ is surjective.
\end{exmp}
The following lemma is immediate from definitions.
\begin{lem}\label{colim}
Let $\mathcal{C}$ and $\mathcal{D}$ be $\infty$-categories with pullbacks and let $\varphi:X\rightarrow Y$ be an effective epimorphism in $\mathcal{C}$. Suppose that $F:\mathcal{C}\rightarrow\mathcal{D}$ is a cocontinuous functor. Then $F(\varphi)$ is an effective epimorphism in $\mathcal{D}$.
\end{lem}
\section{Parametrized Objects in $\infty$-categories}
In this section, we describe the main object of study in this paper: the $\infty$-category of parametrized families of objects in an $\infty$-category. We also develop some general categorical results that we use in subsequent sections.
\begin{defn}
Let $\mathcal{C}$ be an $\infty$-category. Define another $\infty$-category $\mathrm{Fam}(\mathcal{C})$ as follows. The objects of $\mathrm{Fam}(\mathcal{C})$ are pairs $(X,F)$, where $X$ is a small $\infty$-groupoid and $F:X\rightarrow\mathcal{C}$ is a functor. A map $(X_{1},F_{1})\rightarrow(X_{2},F_{2})$ is a pair $(\varphi,\varphi_{\star})$, where $\varphi:X_{1}\rightarrow X_{2}$ is a functor and $\varphi_{\star}$ is a natural transformation \xymatrix@C+2pc{
X \rtwocell^{F_{1}}_{F_{2}\circ\varphi}{\;\;\;\varphi_{\star}} & \mathcal{C}}.\\

One can think of an object of $\mathrm{Fam}(\mathcal{C})$ as a set of objects in $\mathcal{C}$ ``parametrized" by some $\infty$-groupoid/homotopy type $X$.

\begin{exmp}\label{classic}
Let $*$ denote the terminal $\infty$-groupoid. There is an equivalence of categories $\mathrm{Fam}(*)\simeq\mathrm{Grpd}_{\infty}$. This is because we can identify $\mathrm{Fam}(*)$ with the slice category $(\mathrm{Grpd}_{\infty})_{/*}$, which is equivalent to $\mathrm{Grpd}_{\infty}$.
\end{exmp}
%\pagebreak
Example \ref{classic} indicates that our constructions in this paper degenerate to classical notions when considering families of ``points" indexed by $\infty$-groupoids - the data of which essentially constitutes topological spaces. In particular, our construction of the fundamental group of a parametrized family (Definition \ref{fun}) is equivalent to the usual fundamental group construction when considering objects in $\mathrm{Fam}(*)$.

\vspace{0mm}
\end{defn}
\begin{exmp}
If $X$ happens to be a groupoid of the form $\mathbf{B}G$ for a group $G$, then any object $(X,F)\in\mathrm{Fam}(\mathcal{C})$ is just an object of $\mathcal{C}$ equipped with a $G$-action.
\end{exmp}
\begin{rem}
For any $\infty$-category $\mathcal{C}$, $\mathrm{Fam}(\mathcal{C})$ is an extensive $\infty$-category. 
\end{rem}
\begin{term}
For $(X,F)\in\mathrm{Fam}(\mathcal{C})$, we will refer to $X$ as the \textit{shape} of $(X,F)$ and $F$ as the \textit{arrow} of $(X,F)$.
\end{term}
\begin{rem}
There is a fully faithful, left-exact embedding $\sigma:\mathcal{C}\hookrightarrow\mathrm{Fam}(\mathcal{C})$ sending $\mu$ to the family $(*,\gamma)$, where $\gamma:*\rightarrow\mathcal{C}$ is the functor that picks out $X\in\mathcal{C}$.
\end{rem}

\begin{rem}
The $\mathrm{Fam}(-)$ construction extends to an $(\infty,2)$-endofunctor $\mathrm{Fam}(-):\mathrm{Cat}_{\infty}\rightarrow\mathrm{Cat}_{\infty}$. %For $\gamma:X\rightarrow Y$ a map of $\infty$-categories, the induced map $\mathrm{Fam}(\gamma)$ sends a family $(Q,F)\in\mathrm{Fam}(X)$ to the family $(Q,\gamma\circ F)\in\mathrm{Fam}(Y)$.
\end{rem}

\vspace{0mm}
\begin{prop}\label{uni}
$\mathrm{Fam}(\mathcal{C})$ is the universal colimit completion of $\mathcal{C}$ with respect to diagrams indexed by $\infty$-groupoids. More precisely:
\begin{itemize}
    \item Any functor $D:K\rightarrow\mathrm{Fam}(\mathcal{C})$ where $K$ is an $\infty$-groupoid has a colimit in $\mathrm{Fam}(\mathcal{C})$.
    \item Let $\mathcal{D}$ be an $\infty$-category with all $\mathrm{Grpd}_{\infty}$-indexed colimits and denote by $[\mathrm{Fam}(\mathcal{C}),\mathcal{D}]_{\star}$ the full subcategory of $[\mathrm{Fam}(\mathcal{C}),\mathcal{D}]$ spanned by functors which preserve $\mathrm{Grpd}_{\infty}$-indexed colimits. Then there is an equivalence of categories:
    \begin{equation}
    [\mathcal{C},\mathcal{D}]\xrightarrow{\simeq}[\mathrm{Fam}(\mathcal{C}),\mathcal{D}]_{\star}
    \end{equation}
\end{itemize}
\end{prop}

\begin{rem}
Generally, if $\mathcal{C}$ is an $n$-category, then $\mathrm{Fam}(\mathcal{C})$ is the universal colimit completion of $\mathcal{C}$ with respect to diagrams indexed by $(n-1)$-groupoids. The construction is exactly the same except $\mathrm{Fam}(\mathcal{C})$ has objects $(X,F)$ in which $X$ is an $(n-1)$-groupoid and $F:X\rightarrow\mathcal{C}$ is a functor. For example if $\mathcal{C}$ is an ordinary category, then $\mathrm{Fam}(\mathcal{C})$ is its coproduct completion.
\end{rem}
\vspace{0mm}
The following proposition reflects a general principle of colimit completions in ($\infty$-)categories: to form the universal completion of a category $\mathcal{C}$ under colimits of a certain shape, one takes something resembling the closure of representable ($\infty$-)presheaves on $\mathcal{C}$ under colimits of that shape.
\vspace{0mm}
\begin{prop}
Let $[\mathcal{C}^{op},\mathrm{Grpd}_{\infty}]_{\sharp}$ denote the full subcategory of $[\mathcal{C}^{op},\mathrm{Grpd}_{\infty}]$ spanned by colimits of representable $\infty$-prestacks indexed by $\infty$-groupoids and let $y:\mathcal{C}\hookrightarrow[\mathcal{C}^{op},\mathrm{Grpd}_{\infty}]$ denote the Yoneda embedding. Then the functor
\begin{align}
\mathrm{Fam}(\mathcal{C})\rightarrow[\mathcal{C}^{op},\mathrm{Grpd}_{\infty}]_{\sharp}
\end{align}
that sends $(X,F)\mapsto\varinjlim(y\circ F)$ induces an equivalence of categories.
\end{prop}
%\subsection{Categorical properties of $\mathbf{Fam}(\mathcal{C})$}
%\vspace{0mm}In this section we prove some general categorical properties of parametrized objects that we will use in subsequent sections.

\vspace{0mm}
We now give an explicit construction of (co)limits in $\mathrm{Fam}(\mathcal{C})$. In this construction and in Lemmae \ref{lims} and \ref{connect}, we will use some notation introduced in \S4, in particular $\Pi_{\infty}$.
\begin{proof}[Construction of (co)limits in $\mathrm{Fam}(\mathcal{C})$]
Fix an $\infty$-category $K$ and let $D:K\rightarrow\mathrm{Fam}(\mathcal{C})$ be a diagram. The shape $\Pi_{\infty}(\varinjlim(D))$ of $\varinjlim(D)$ (when it exists) is given by:
\begin{align}
\Pi_{\infty}(\varinjlim(D))=\varinjlim(K\xrightarrow{D}\mathrm{Fam}(\mathcal{C})\xrightarrow{\Pi_{\infty}}\mathrm{Grpd}_{\infty})
\end{align}
 By definition, each $x\in \Pi_{\infty}\circ D(K)$ is equipped with a map $\Pi_{\infty}(D(\beta))\rightarrow\mathcal{C}$ which commutes with the required triangles. Thus there is  induced a functor:
  \begin{align}
 \varinjlim(K\xrightarrow{D}\mathrm{Fam}(\mathcal{C})\xrightarrow{\Pi_{\infty}}\mathrm{Grpd}_{\infty})\rightarrow\mathcal{C}
 \end{align}
 
 This is precisely the arrow of $\varinjlim(D)$. Now we describe limits\footnote{Thanks to MathOverflow user Kyle Ferendo for describing this to us.}. As before, let $D:K\rightarrow\mathrm{Fam}(\mathcal{C})$ be a diagram. the shape $\Pi_{\infty}(\varprojlim(D))$ of $\varprojlim(D)$ is given by the limit:
\begin{align}
\Pi_{\infty}(\varprojlim(D))=\varprojlim(K\xrightarrow{D}\mathrm{Fam}(\mathcal{C})\xrightarrow{\Pi_{\infty}}\mathrm{Grpd}_{\infty})
\end{align}
For each $x\in K$, there is a canonical projection $p_{x}:\varprojlim(\Pi_{\infty}\circ D)\rightarrow \Pi_{\infty}\circ D(x)$. But by definition, for such $x$ we have functors $\gamma_{x}:x\rightarrow\mathcal{C}$, so we can define a natural functor $\zeta:K\rightarrow[\varprojlim(\Pi_{\infty}\circ D),\mathcal{C}]$ by $\zeta(x)=\gamma_{x}\circ p_{x}:\varprojlim(\Pi_{\infty}\circ D)\rightarrow \mathcal{C}$. The arrow of $\varprojlim(D)$ is given by $\varprojlim(\zeta)$.
\end{proof}

\begin{rem}
Actually, the colimit in (3) must be taken in $\mathrm{Cat}_{\infty}$ in order for the desired universal property to kick in ($\mathcal{C}$ is not necessarily an $\infty$-groupoid), but this does not affect anything.
\end{rem}

\begin{lem}\label{lims}
Fix an $\infty$-category $\mathcal{C}$ and a small $\infty$-category $\lambda$. Suppose limits indexed by $\lambda$ exist in $\mathcal{C}$. Then limits indexed by $\lambda$ exist in $\mathrm{Fam}(\mathcal{C})$. 
\end{lem}

\begin{proof}
This follows from the fact that the limit of a functor $D:\lambda\rightarrow\mathrm{Fam}(\mathcal{C})$ is computed as a combination of $\lambda$-indexed limits of $\infty$-groupoids (which are guaranteed to exist) and $\lambda$-indexed limits in the functor category $[\varprojlim(\Pi_{\infty}\circ D),\mathcal{C}]$. The claim holds since limits in $[\varprojlim(\Pi_{\infty}\circ D),\mathcal{C}]$ are computed point-wise as limits in $\mathcal{C}$.
\end{proof}
%\begin{lem}
%The functor $\mathrm{Fam}:\mathrm{Cat}_{\infty}\rightarrow\mathrm{Cat}_{\infty}$ preserves left-exactness of functors.
%\end{lem}
%\begin{proof}
%Fix $\infty$-categories $X$ and $Y$ and let $\varphi:X\rightarrow Y$ be a left-exact functor. 
%\end{proof}

\begin{defn}

Let $\mathbb{H}$ be an $\infty$-category. Recall that an object $X\in\mathbb{H}$ is \textit{n-truncated} if the mapping $\infty$-groupoids $\mathrm{Hom}_{\mathbb{H}}(Y,X)$ are $n$-groupoids for all $Y\in\mathbb{H}$. The \textit{n-truncation} functor $\tau_{\leq n}:\mathbb{H}\rightarrow\tau_{\leq n}\mathbb{H}$ is left adjoint to the full inclusion $\tau_{\leq n}\mathbb{H}\hookrightarrow\mathbb{H}$ of $n$-truncated objects in $\mathbb{H}$.
\end{defn}
\begin{prop}
Let $\mathcal{C}$ be an $\infty$-category. An object $(X,F)\in\mathrm{Fam}(\mathcal{C})$ is $0$-truncated if and only if $X$ is a discrete $\infty$-groupoid/set.
\end{prop}
\begin{lem}\label{connect}    
Let $\mathcal{C}$ be an $\infty$-category. An object $(X,F)\in\mathrm{Fam}(\mathcal{C})$ is connected if and only if $X$ is a connected $\infty$-groupoid.             
\end{lem}
\begin{proof}
Since $\mathrm{Fam}(\mathcal{C})$ is extensive, this reduces to showing that in any coproduct decomposition $(X,F)\simeq(X_{1},F_{1})\coprod(X_{2},F_{2})$, exactly one of the $(X_{i},F_{i})$ is not initial if and only if $X$ is connected as an $\infty$-groupoid. But this is immediate since $\Pi_{\infty}((X_{1},F_{1})\coprod(X_{2},F_{2}))=X_{1}\coprod X_{2}$.
\end{proof}

%\begin{proof}
%Let $*$ denote the terminal object of $\mathrm{Fam}(\mathcal{C})$ and let $X$ be a disconnected $\infty$-groupoid. We can write $X\simeq\coprod_{i\in I}X_{i}$, where each $X_{i}$ is a connected inhabited $\infty$-groupoid and $|I|\geq 2$. Since the mapping space $\mathrm{Hom}_{\mathrm{Fam}(\mathcal{C})}(Y,*)$ is contractible for every $Y\in\mathrm{Fam}(\mathcal{C})$, for any functor $X\xrightarrow{F}\mathcal{C}$, we have that  $\mathrm{Hom}_{\mathrm{Fam}(\mathcal{C})}((X,F),*)\coprod\mathrm{Hom}_{\mathrm{Fam}(\mathcal{C})}((X,F),*)\simeq*\coprod *$. However, since $|I|\geq 2$, it is clear that $|\pi_{0}\mathrm{Hom}_{\mathrm{Fam}(\mathcal{C})}((X,F),*\coprod*))|>2$. Thus $(X,F)$ must be be disconnected.\\

    %Conversely, suppose that $(X,F)\in\mathrm{Fam}(\mathcal{C})$ is disconnected. Let $(Y,G)=\coprod_{i\in I}(Y_{i},G_{i})\in\mathrm{Fam}(\mathcal{C})$ be an object such that $\mathrm{Hom}_{\mathrm{Fam}(\mathcal{C})}((X,F),\coprod_{i\in I}(Y_{i},G_{i}))\cancel{\simeq}\coprod_{i\in I}\mathrm{Hom}_{\mathrm{Fam}(\mathcal{C})}((X,F),(Y_{i},G_{i}))$ for some set $I$ and suppose that there is an equivalence of $\infty$-groupoids $\mathrm{Hom}_{\mathrm{Grpd}_{\infty}}(X,\coprod_{i\in I}Y_{i})\simeq\coprod_{i\in I}\mathrm{Hom}_{\mathrm{Grpd}_{\infty}}(X,Y_{i})$.
%\end{proof}
This implies that any object in $\mathrm{Fam}(\mathcal{C})$ can be written (essentially uniquely) as a coproduct of connected objects. Thus, we may regard $\mathrm{Fam}(\mathcal{C})$ for $\mathcal{C}$ an $\infty$-category in much the same way as we treat the 1-categorical case (since $\mathrm{Fam}(\mathcal{C})$ for $\mathcal{C}$ a 1-category is a coproduct completion, every object can be written as a coproduct of connected objects). This property of every object admitting coproduct decomposition into connected objects is of course shared by the category $\mathfrak{Top}$ of topological spaces and continuous maps.\\

\section{The Fundamental $\infty$-groupoid of a parametrized family}
In this section, we study phenomena pertaining to the homotopy theory of parametrized families in $\infty$-categories. In particular, we define the fundamental $(\infty$-)group(oid)s of objects in $\mathrm{Fam}(\mathcal{C})$. %This generalizes the constructions in the previous subsection to the $\infty$-categorical context. 
\begin{defn}
Let $\mathcal{C}$ be an $\infty$-category and consider $\mathrm{Fam}(\mathcal{C})$. The \textit{fundamental $\infty$-groupoid} functor $\Pi_{\infty}:\mathrm{Fam}(\mathcal{C})\rightarrow\mathrm{Grpd}_{\infty}$ is the functor sending $(X,F)\mapsto X$. Equivalently, it is the Grothendieck construction $\Pi_{\infty}\simeq\int[-,\mathcal{C}]$ of the representable prestack $[-,\mathcal{C}]:\mathrm{Grpd}_{\infty}^{op}\rightarrow\mathrm{Cat}_{\infty}$ that sends $X\mapsto[X,\mathcal{C}]$. 
\end{defn}

%\begin{exmp}
%Let $\mathrm{Fam}(\mathcal{C})_{\mathbf{B}(-)}$ denote the full subcategory of $\mathrm{Fam}(\mathcal{C})$ spanned by $\mathcal{C}$-valued permutation representations of groups (objects with domain $\mathbf{B}G$ for a group $G$) and let $\mathrm{Grpd}_{\leq 1}$ denote the category of 1-object groupoids. Then $\Pi_{\infty}$  restricts to a precosheaf of groups $\Pi_{\infty}|_{\mathrm{Fam}(\mathcal{C})_{\mathbf{B}(-)}}:\mathrm{Fam}(\mathcal{C})_{\mathbf{B}(-)}\rightarrow\mathrm{Grpd}_{\leq 1}\xrightarrow{\simeq}\mathcal{G}\mathrm{rp}$ sending $(\mathbf{B}G,F)\mapsto G$.

%\end{exmp}
\vspace{0mm}
\begin{rem}
%In general, $\Pi_{\infty}$ sends an object in the colimit completion of an $n$-category with respect to diagrams of $(n-1)$-groupoids to a homotopy $(n-1)$-type. For example, when $\mathcal{C}$ is a 1-category, $\Pi_{\infty}$ reduces to the family fibration [2, Chapter 6.1], regarding a set as a homotopy 0-type.
Let $\tau_{\leq 0}\mathrm{Fam}(\mathcal{C})$ denote the full subcategory of $\mathrm{Fam}(\mathcal{C})$ spanned by $0$-truncated objects, i.e families $(X,F)$ such that $X$ is a set. The inclusion $i:\tau_{\leq 0}\mathrm{Fam}(\mathcal{C})\subset\mathrm{Fam}(\mathcal{C})$ induces a commutative square of $\infty$-categories:

\begin{displaymath}
\xymatrixcolsep{2.0cm}
\xymatrixrowsep{1.1cm}
\xymatrix{
\mathrm{Fam}(\mathcal{C})  \ar[r]^{\Pi_{\infty}} &\mathrm{Grpd}_{\infty}\\
\tau_{\leq 0}\mathrm{Fam}(\mathcal{C})\ar@{^{(}->}[u]^{i} \ar[r]^{\Pi_{\infty}|_{\tau_{\leq 0}\mathrm{Fam}(\mathcal{C})}}
&\mathrm{Set}\ar@{^{(}->}[u]_{i}}
\end{displaymath}
\end{rem}
The bottom horizontal restriction map is the classical family fibration [2, Chapter 6.1].

\begin{prop}\label{adj}
Let $\mathcal{C}$ be an $\infty$-category with terminal object $*$. Define the functor $\Delta:\mathrm{Set}\rightarrow\mathrm{Fam}(\mathcal{C})$  by $\Delta:I\mapsto\coprod_{i\in I}\sigma(*)$, where $\sigma:\mathcal{C}\hookrightarrow\mathrm{Fam}(\mathcal{C})$ is the singleton embedding. Then the canonical restriction $\Pi_{\infty}|_{\tau_{\leq 0}\mathrm{Fam}(\mathcal{C})}:{\tau_{\leq 0}}\mathrm{Fam}(\mathcal{C})\rightarrow\mathrm{Set}$ of $\Pi_{\infty}$ to the full subcategory of 0-truncated objects in $\mathrm{Fam}(\mathcal{C})$ fits into an adjunction:
\begin{displaymath}
( \Pi_{\infty}|_{\tau_{\leq 0}\mathrm{Fam}(\mathcal{C})}\dashv\Delta ) 
  \;\;\; : \;\;\;
  \tau_{\leq 0}\mathrm{Fam}(\mathcal{C})
    {\stackrel{\overset{\Delta}{\longleftarrow}}{\underset{\Pi_{\infty}}{\longrightarrow}}}
  \mathrm{Set}
  \,.
\end{displaymath}
\end{prop}
\begin{proof}
%By \cite{bor}, Remark 6.2.2, the $\mathrm{Fam}$ construction extends to a strict 2-functor $\mathrm{Fam}(-):\mathrm{Cat}\rightarrow\mathrm{Cat}$. Note that since $\mathcal{C}$ has a terminal object we have a canonical adjunction $\mathcal{C}\rightarrow*$
Fix $(X,F)\in\mathrm{Fam}(\mathcal{C})$ and $I\in\mathrm{Set}$ and denote by $\coprod_{j\in\pi_{0} X}(X_{j},F_{j})$ the canonical coproduct decomposition of $(X,F)$ into connected objects (which exists by virtue of Lemma 3.15). Applying definitions, we have:
\begin{align}
\mathrm{Hom}_{\tau_{\leq 0}\mathrm{Fam}(\mathcal{C})}((X,F),\Delta I)\simeq\mathrm{Hom}_{\tau_{\leq 0}\mathrm{Fam}(\mathcal{C})}(\coprod_{j\in \pi_{0}X}(X_{j},F_{j}),\coprod_{i\in I}\sigma(*))
\end{align}
Since each $(X_{i},F_{i})$ is connected and representables send colimits to limits in their first argument:
\begin{align}
\mathrm{Hom}_{\tau_{\leq 0}\mathrm{Fam}(\mathcal{C})}(\coprod_{j\in \pi_{0}X}(X_{j},F_{j}),\coprod_{i\in I}\sigma(*))\simeq\prod_{j\in \pi_{0}X}\mathrm{Hom}_{\tau_{\leq 0}\mathrm{Fam}(\mathcal{C})}((X_{j},F_{j}),\coprod_{i\in I}\sigma(*))
\end{align}
\begin{align}
    \simeq\prod_{j\in \pi_{0}X}\coprod_{i\in I}\mathrm{Hom}_{\tau_{\leq 0}\mathrm{Fam}(\mathcal{C})}((X_{j},F_{j}),\sigma(*))
\end{align}
Since $\sigma$ is left-exact, $\sigma(*)$ is terminal in $\tau_{\leq 0}\mathrm{Fam}(\mathcal{C})$. So we have:
\begin{align}
\prod_{j\in \pi_{0}X}\coprod_{i\in I}\mathrm{Hom}_{\tau_{\leq 0}\mathrm{Fam}(\mathcal{C})}((X_{j},F_{j}),\sigma(*))\simeq\prod_{j\in \pi_{0}X}\coprod_{i\in I}*
\end{align}
\begin{align}
\simeq\prod_{j\in \pi_{0}X}\mathrm{Hom}_{\mathrm{Set}}(*,I)
\end{align}
\begin{align}
\simeq\mathrm{Hom}_{\mathrm{Set}}(\coprod_{j\in \pi_{0}X}*,I)
\end{align}
\begin{align}
\simeq\mathrm{Hom}_{\mathrm{Set}}(\Pi_{\infty}|_{\tau_{\leq 0}\mathrm{Fam}(\mathcal{C})}(X,F),I)
\end{align}
\end{proof}
We observe that Proposition \ref{adj} implies the following statement.
\begin{cor}
\textit{Let $\mathcal{C}$ be an $\infty$-category such that $\mathrm{Fam}(\mathcal{C})$ is an $\infty$-topos (e.g $\mathcal{C}=\mathrm{Grpd}^{*/}, \mathbf{Sp}$). Then the full subcategory $\tau_{\leq0}\mathrm{Fam}(\mathcal{C})$ of $\mathrm{Fam}(\mathcal{C})$ on $0$-truncated objects is a locally connected 1-topos}. 
\end{cor}
Let $X$ be a topological space. Recall that its fundamental group $\pi_{1}(X,x)$ at basepoint $x\in X$ can be obtained by the automorphism group $\mathrm{Aut}_{\Pi_{1}(X)}(x)$ of $x$ in its fundamental groupoid $\Pi_{1}(X)$, where $x$ is regarded as an object of $\Pi_{1}(X)$. In the rest of this section, we will describe a natural construction of the fundamental groups of parametrized families in an $\infty$-category based on this perspective.

\vspace{0mm}

\vspace{0mm}
\begin{defn}
Let $\mathcal{C}$ be an $\infty$-category. Define the \textit{fundamental groupoid} functor $\Pi_{1}:\mathrm{Fam}(\mathcal{C})\rightarrow\mathrm{Grpd}$ by $\pi_{1}=\tau_{\leq1}\circ\Pi_{\infty}$.
\end{defn}
\begin{prop}
Let $\mathcal{C}$ be an $\infty$-category. By abuse of notation, denote by $*$ both the terminal object of $\mathcal{C}$ and the terminal $(\infty$-)groupoid. Then $\Pi_{1}$ induces a functor
\begin{align}
\Pi_{1}^{+}:\mathrm{Fam}(\mathcal{C})^{*/}\rightarrow\mathrm{Grpd}^{*/}
\end{align}
on categories of pointed objects, i.e $\Pi_{1}$ takes pointed families in $\mathcal{C}$ to pointed groupoids. Furthermore, $\Pi_{1}^+$ preserves small colimits.
\end{prop}

\begin{proof}
Fix a pointed object $*\xrightarrow{x}(X,F)$ in $\mathrm{Fam}(\mathcal{C})$. That $\Pi_{1}$ induces a functor on categories of pointed objects is immediate from the observation that both $\Pi_{\infty}$ and $\tau_{\leq 1}$ preserve terminal objects, hence $\Pi_{1}=\tau_{\leq 1}\circ\Pi_{\infty}$ induces a map $*\xrightarrow{\Pi_{1}(x)}\Pi_{1}(X,F)$.\\

From the construction of colimits in $\mathrm{Fam}(\mathcal{C})$, it is clear that $\Pi_{\infty}$ preserves them. Since $\tau_{\leq1}$ is a left adjoint, it also preserves colimits. By definition, for $K$ a small $\infty$-category and $D:K\rightarrow\mathrm{Fam}(\mathcal{C})^{*/}$ a diagram, we have $\varinjlim(D)=\varinjlim(U\circ D)\coprod*\leftarrow*$, where $U:\mathrm{Fam}(\mathcal{C})^{*/}\rightarrow\mathrm{Fam}(\mathcal{C})$ is the canonical projection. But since the composite $\Pi_{1}=\tau_{\leq 1}\circ\Pi_{\infty}$ preserves colimits and terminal objects, we have:
\begin{align}
\Pi_{1}(\varinjlim(D))=\Pi_{1}(\varinjlim(U\circ D)\coprod*\leftarrow*)
\end{align}
\begin{align}
\simeq\Pi_{1}(\varinjlim(U\circ D))\coprod*\leftarrow*
\end{align}
\begin{align}
\simeq\varinjlim(U\circ\Pi_{1}\circ D)\coprod*\leftarrow*
\end{align}
Coupling the fact that $\Pi_{1}$ evidently commutes with $U$ with the definition of colimits in $\mathrm{Fam}(\mathcal{C})^{*/}$, we get:
\begin{align}
((\varinjlim(U\circ\Pi_{1}\circ D)\coprod*\leftarrow*)\simeq\varinjlim(\Pi_{1}\circ D)
\end{align}
Thus the proposition follows.
\end{proof}

\begin{defn}\label{fun}
The \textit{fundamental group} $\pi_{1}((X,F),x)$ at $x\in(X,F)$ of a pointed family $*\xrightarrow{x}(X,F)\in\mathrm{Fam}(\mathcal{C})^{*/}$ is the automorphism group $\pi_{1}((X,F),x)=\mathrm{Aut}_{\tau_{\leq 1}(\Pi_{\infty}(X,F))}(x)$. This extends to a functor $\pi_{1}:\mathrm{Fam}(\mathcal{C})^{*/}\rightarrow \mathcal{G}\mathrm{rp}$. Equivalently, it is the based fundamental group of the pointed $\infty$-groupoid $(\Pi_{\infty}(X,F),x)$ regarded as a pointed topological space under the homotopy hypothesis.
\end{defn}
\begin{rem}
$\pi_{1}((X,F),x)$ can also be computed as the first based simplicial homotopy group $\pi_{1}(N(\Pi_{1}(X,F)),x)$ of the pointed Kan complex $N(\Pi_{1}(X,F)))$ at $x\in N(\Pi_{1}(X,F))_{0}$. 
\end{rem}
\vspace{0mm}
\begin{rem}
Via the equivalence $\mathrm{Fam}(*)\simeq\mathrm{Grpd}_{\infty}$, the $\pi_{1}$ construction of Definition \ref{fun} recovers the classical fundamental group of a topological space.  
\end{rem}
%We can now import basic results about the classical $\pi_1$ into our new setting.
\begin{prop}
Let $\mathcal{C}$ be an $\infty$-category and fix a connected object $(X,F)\in\mathrm{Fam}(\mathcal{C})$. Let $(x_{0},\phi_{0}), (x_{1},\phi_{1}):*\rightarrow(X,F)$ be two basepoints in $(X,F)$. Then there is a canonical isomorphism of groups $\pi_{1}((X,F),x_{0})\xrightarrow{\simeq}\pi_{1}((X,F),x_{1})$.
\end{prop}
\begin{proof}
By Lemma \ref{connect}, $X$ (and hence $\tau_{\leq 1}(X)$) is a connected $\infty$-groupoid. Since $\tau_{\leq 1}(X)$ is connected, there are canonical equivalences of groupoids
\begin{equation}
\mathbf{B}\mathrm{Aut}_{\tau_{\leq 1}(X)}(x_{0}(*))\xrightarrow{\simeq}\tau_{\leq1}(X)\xleftarrow{\simeq}\mathbf{B}\mathrm{Aut}_{\tau_{\leq 1}(X)}(x_{1}(*))
\end{equation}.

Since $\mathbf{B}\mathrm{Aut}_{\tau_{\leq 1}(X)}(x_{0}(*))$ and $\mathbf{B}\mathrm{Aut}_{\tau_{\leq 1}(X)}(x_{1}(*))$ are equivalent and have one object each, they must be isomorphic. The proposition follows from the fact that $\mathbf{B}$ is fully faithful and hence reflects isomorphisms.
\end{proof}
\begin{comment}

\begin{prop}\label{colimit}
Define the ``based fundamental groupoid" functor $\Pi_{1}:\mathrm{Fam}(\mathcal{C})^{*/}\rightarrow\mathrm{Grpd}^{*/}$ by $\Pi_{1}=\tau_{\leq 1}\circ\Pi_{\infty}$. Then $\Pi_{1}$ preserves colimits.
\end{prop}

\begin{proof}
From the construction of colimits in $\mathrm{Fam}(\mathcal{C})$, it is clear that the (unbased) $\Pi_{\infty}$ preserves them. Since $\tau_{\leq1}$ is a left adjoint, it also preserves colimits. For $D:K\rightarrow\mathrm{Fam}(\mathcal{C})^{*/}$ a diagram, we have $\varinjlim(D)=\varinjlim(U\circ D)\coprod*\leftarrow*$, where $U:\mathrm{Fam}(\mathcal{C})^{*/}\rightarrow\mathrm{Fam}(\mathcal{C})$ is the canonical projection. But since the (unbased) composite $\tau_{\leq 1}\circ\Pi_{\infty}$ must preserve colimits and the terminal object, we have:
\begin{align}
\tau_{\leq 1}\circ\Pi_{\infty}(\varinjlim(D))=\tau_{\leq 1}\circ\Pi_{\infty}(\varinjlim(U\circ D)\coprod*\leftarrow*)
\end{align}
\begin{align}
\simeq\tau_{\leq 1}\circ\Pi_{\infty}(\varinjlim(U\circ D))\coprod*\leftarrow*
\end{align}
\begin{align}
\simeq\varinjlim(U\circ\tau_{\leq 1}\circ\Pi_{\infty}\circ D)\coprod*\leftarrow*
\end{align}
Coupling the fact that $\tau_{\leq 1}\circ\Pi_{\infty}$ evidently commutes with $U$ with the definition of colimits in $\mathrm{Fam}(\mathcal{C})^{*/}$, we get:
\begin{align}
((\varinjlim(U\circ\tau_{\leq 1}\circ\Pi_{\infty}\circ D)\coprod*\leftarrow*)\simeq\varinjlim(\tau_{\leq 1}\circ\Pi_{\infty}\circ D)
\end{align}
\end{proof}
\end{comment}
\section{A Grothendieck Topology on Families in an $\infty$-category}
Given an $\infty$-category $\mathcal{C}$, we can treat objects in $\mathrm{Fam}(\mathcal{C})$ as ``spaces" inside of $\mathcal{C}$. Thus, it is natural to ask for an appropriate notion of an open covering of a family $(X,F)\in\mathrm{Fam}(\mathcal{C})$. This can be accomplished by defining a \textit{Grothendieck topology} on $\mathrm{Fam}(\mathcal{C})$. In this section, we endow $\mathrm{Fam}(\mathcal{C})$ the structure of a site based on Carchedi's epimorphism topology on $\infty$-topoi (Example \ref{effe}). We start with a definition.

\begin{defn}
Let $I$ be a set and let $\mathbb{H}$ be an $\infty$-topos. A family of maps of $\{X_{i}\xrightarrow{f_{i}}X\}_{i\in I}$ in $\mathbb{H}$ is an \textit{effective epimorphic family} if the induced map $\coprod_{i\in I}f_{i}:\coprod_{i\in I}X_{i}\rightarrow X$ is an effective epimorphism.
\end{defn}
The following theorem asserts that a form of the epimorphism topology (See Example \ref{effe}) on $\mathrm{Grpd}_{\infty}$ holds in the context of families in an $\infty$-category.
\begin{thm}\label{gr}
Let $\mathcal{C}$ be an $\infty$-category with pullbacks and let $(X,F)$ be an object of $\mathrm{Fam}(\mathcal{C})$. Define a family of maps $\{(X_{i},F_{i})\xrightarrow{(f_{i},\varphi_{i})}(X,F)\}_{i\in I}$ with codomain $(X,F)$ to be a \textit{covering family} if the induced family 
\begin{align}
\{\Pi_{\infty}(X_{i},F_{i})\xrightarrow{\Pi_{\infty}(f_{i},\varphi_{i})}\Pi_{\infty}(X,F)\}_{i\in I}
\end{align}
\begin{align}
=\{X_{i}\xrightarrow{f_{i}}X\}_{i\in I}
\end{align}
is an effective epimorphic family of $\infty$-groupoids. These covering families define a Grothendieck topology on $\mathrm{Fam}(\mathcal{C})$.
\end{thm}
\begin{proof}
%We need to show that the conditions of [3, Definition 2.2.3] are satisfied.
Clearly, any equivalence $\{(\gamma,\gamma_{\star}):(X',F')\xrightarrow{\simeq}(X,F)\}$ is a covering since it must hold that the underlying map of $\infty$-groupoids $\{\gamma:X'\rightarrow X\}$ is also an equivalence. Coupling Lemma \ref{lims} with the assumption that $\mathcal{C}$ has pullbacks implies that for any covering family $\{(X_{i},F_{i})\xrightarrow{(f_{i},\varphi_{i})}(X,F)\}_{i\in I}$ and a map $(g,\phi):(X',F')\rightarrow(X,F)$, there exists a pullback square:
\begin{displaymath}\xymatrixcolsep{1.7cm}
    \xymatrix{
        (\coprod_{i\in I}(X_{i},F_{i}))\times_{(X,F)}(X',F') \ar[r]\pushoutcorner  \ar[d] & (X',F')\ar[d]^{(g,\phi)} \\
        \coprod_{i\in I}(X_{i},F_{i}) \ar[r]_{\coprod_{i\in I}(f_{i},\varphi_{i})}       & (X,F) }
\end{displaymath}
in $\mathrm{Fam}(\mathcal{C})$ for each $i\in I$. We claim that $\{(X_{i},F_{i})\times_{(X,F)}(X',F')\rightarrow(X',F')\}_{i\in I}$ is a covering family, i.e that the induced map $\coprod_{i\in I}(X_{i}\times_{X}X')\rightarrow X'$ is an effective epimorphism. $\Pi_{\infty}$ evidently preserves small (co)limits, so applying $\Pi_{\infty}$ induces a pullback square of the underlying $\infty$-groupoids:
\begin{displaymath}\xymatrixcolsep{1.5cm}
    \xymatrix{
        (\coprod_{i\in I}X_{i})\times_{X}X' \ar[r]\pushoutcorner  \ar[d] & X'\ar[d]^{g} \\
        \coprod_{i\in I}X_{i} \ar[r]_{\coprod_{i\in i}f_{i}}       & X }
\end{displaymath}
for each $i\in I$. By assumption, the bottom horizontal map is an effective epimorphism, so the top horizontal map is also an effective epimorphism by [6, Proposition 6.2.3.15]. The claim then follows from the fact that coproducts of $\infty$-groupoids are universal, so that $(\coprod_{i\in I}X_{i})\times_{X}X'\simeq\coprod_{i\in I}(X_{i}\times_{X}X'\rightarrow X')$ is an effective epimorphism. That covering families are stable under composition are stable under composition is clear from [6, Corollary 7.2.1.12], so we are done.
\end{proof}
We will refer to the Grothendieck topology of Theorem \ref{gr} as the \textit{effective topology} and denote the associated $\infty$-site as $(\mathrm{Fam}(\mathcal{C}),E)$. Denote the epimorphism topology [3, Definition 2.2.5] on $\mathrm{Grpd}_{\infty}$ by $(\mathrm{Grpd}_{\infty}, Epi)$. By construction, $\Pi_{\infty}$ yields a morphism of sites:
\begin{equation}
\Pi_{\infty}:(\mathrm{Fam}(\mathcal{C}),E)\rightarrow(\mathrm{Grpd}_{\infty},Epi)
\end{equation}
\begin{rem}
The effective topology on $\mathrm{Fam}(*)\simeq\mathrm{Grpd}_{\infty}$ is precisely the epimorphism topology.
\end{rem}

\begin{prop}
Let $\mathcal{C}$ be an $\infty$-category with pullbacks such that $\mathrm{Fam}(\mathcal{C})$ is an $\infty$-topos (e.g $\mathcal{C}=\mathbf{Sp}$). Then the effective topology on $\mathrm{Fam}(\mathcal{C})$ contains the epimorphism topology, i.e every covering family in the epimorphism topology is also an effective covering.
\end{prop}
\begin{proof}
We need to show that if $\{(f_{i},\varphi_{i}):(X_{i},F_{i})\rightarrow(X,F)\}_{i\in I}$ is an effective epimorphic family, then the induced family of maps of $\infty$-groupoids $\{f_{i}:X_{i}\rightarrow X\}_{i\in I}$ under $\Pi_{\infty}$ is also effective epimorphic. This follows from coupling Lemma \ref{colim} with the fact that $\Pi_{\infty}$ preserves small colimits.
\end{proof}

\section{Appendix A: (Higher) Loci and Miscellaneous Results}
The fundamental group of a parametrized family (definition \ref{fun}) in many ways looks like the geometric homotopy groups of objects an $\infty$-topos. In fact, there are cases of $\infty$-categories $\mathcal{C}$ such that $\mathrm{Fam}(\mathcal{C})$ is a topos. The restriction of our definition of the fundamental $\infty$-groupoids/groups of parametrized families to families of objects in such categories thus behaves similarly to notions of homotopical invariants of objects in higher topoi. In this section, we discuss such categories in both the ordinary and higher categorical case and state some of our results relevant to the topic.
\begin{defn}[Joyal\footnote{Joyal also requires loci to be pointed, i.e to have a zero object.}]
A \textit{locus} is a locally presentable category $\mathcal{C}$ such that $\mathrm{Fam}(\mathcal{C})$ is a topos. 
\end{defn}
In [5], Joyal notes the following:
\begin{obs}\label{j}
\textit{The category $\mathrm{Set}^{*/}$ of pointed sets is a locus}.
\begin{notn}
For convenience, we will write a family in $\mathcal{C}$ indexed over a set $I$ as $\langle X_{i}\rangle_{i\in I}$, where each $X_{i}\in\mathcal{C}$.
\end{notn}
\end{obs}
\begin{proof}[Proof of Observation \ref{j}]\label{walk}
It is clear that $\mathrm{Set}^{*/}$ is locally presentable, since co-slice categories of locally presentable categories are themselves locally presentable. It remains to show that $\mathrm{Fam}(\mathrm{Set}^{*/})$ is a topos.\\
Let $\widehat{\Delta[1]}$ denote the ``walking-arrow-equipped-with-a-section," i.e the free category on the directed graph 
\xymatrix{
[0]\ar@/^.5pc/[r]^{s} & [1]\ar@/^.5pc/[l]^{r}
} subject to the condition $r\circ s=\mathrm{id}_{[0]}$. We claim that there is an equivalence of categories $\mathrm{Fam}(\mathrm{Set}^{*/})\simeq[\widehat{\Delta[1]},\mathrm{Set}]$. We construct a functor\footnote{This part of the equivalence is constructed in \cite{joy}.} $F:[\widehat{\Delta[1]},\mathrm{Set}]\rightarrow\mathrm{Fam}(\mathrm{Set}^{*/})$ as follows. Fix an object $X\in[\widehat{\Delta[1]},\mathrm{Set}]$. For every $p\in X[0]$, the fiber $X[r]^{-1}(p)$ is canonically pointed by $X[s](p)$ (this is guaranteed to land in $X[r]^{-1}(p)$ because of the condition $X(r)\circ X(s)=\mathrm{id}_{X[0]}$). Now we define: 
\begin{equation}
F:X\mapsto\langle (X[r]^{-1}(p),X[s](p))\rangle_{p\in X[0]}
\end{equation}

 We now construct a pseudo-inverse $G:\mathrm{Fam}(\mathrm{Set}^{*/})\rightarrow[\widehat{\Delta[1]},\mathrm{Set}]$. Let $U:\mathrm{Set}^{*/}\rightarrow\mathrm{Set}$ denote the canonical projection and fix $\langle X_{i}\rangle_{i\in I}\in\mathrm{Fam}(\mathrm{Set}^{*/})$. Define the function $\gamma_{1}:\coprod_{i\in I}U(X_{i})\rightarrow I$ that sends all $x\in X_{i}$ to $i\in I$. There is a canonical function $\gamma_{2}:I\rightarrow\coprod_{i\in I}U(X_{i})$ that sends $i\in I$ to the basepoint of $X_{i}$. Clearly, $\gamma_{1}\circ\gamma_{2}=\mathrm{id}_{I}$, so we can define the functor
\begin{displaymath}
G:\langle X_{i}\rangle_{i\in I}\mapsto\Bigg[
\xymatrix{
G(\langle X_{i}\rangle_{i\in I})[0]=I\ar@/^1.2pc/[r]^{G(\langle X_{i}\rangle_{i\in I})[s]=\gamma_{2}} & G(\langle X_{i}\rangle_{i\in I})[1]=\coprod_{i\in I}U(X_{i})\ar@/^1.2pc/[l]^{G(\langle X_{i}\rangle_{i\in I})[r]=\gamma_{1}}
}\Bigg]
\end{displaymath}
It is straightforward to verify that $G\circ F$ (resp. $F\circ G$) is naturally isomorphic to $\mathrm{id}_{[\widehat{\Delta[1]},\mathrm{Set}]}$ (resp. $\mathrm{id}_{\mathrm{Fam}(\mathrm{Set}^{*/})}$), so the claim is proven. Since $\mathrm{Fam}(\mathrm{Set}^{*/})$ is equivalent to a category of diagrams in $\mathrm{Set}$, it is a topos.
\end{proof}
\begin{cor}\label{cot}
\textit{$[\widehat{\Delta[1]},\mathrm{Set}]$ is a locally connected topos. Furthermore, for any precosheaf $X\in[\widehat{\Delta[1]},\mathrm{Set}]$, there is an isomorphism}
\begin{equation}
\pi_{0}(X)\simeq X[0]
\end{equation}

\end{cor}
\begin{rem}
A similar argument can be used to show that the coproduct completion of the category of pointed objects in a topos $\mathbb{E}$ is itself a topos. This is clear once we note that the condition $r\circ s=\mathrm{id}_{[0]}$ implies that for every object $X\in[\widehat{\Delta[1]},\mathbb{E}]$ and any global element $p:*\rightarrow X[0]$, we get a commutative diagram in $\mathbb{E}$:\\
\begin{displaymath}
\xymatrixcolsep{1.5cm}
    \xymatrix{
    \ast\ar@/^.7pc/[drr]^{\mathrm{id}}\ar@/_.7pc/[ddr]_{X[s]\circ p}\ar@{.>}[dr] & &\\
   &  X[1]\times_{X[0]}\ast\ar[r]\ar[d] & \ast\ar[d]^{p} \\
        &   X[1]\ar[r]^{X[r]}   & X[0] }
\end{displaymath}\\

in which the square is a pullback and the global element $*\rightarrow X[1]\times_{X[0]}\ast$ is induced by universal property, i.e the fibered product $X[1]\times_{X[0]}\ast$ is canonically pointed. By repeating this construction for each $p\in\mathrm{Hom}_{\mathbb{E}}(*,X[0])$, we can build a family of pointed objects in $\mathbb{E}$ by taking fibers. This assignment is functorial and yields an equivalence of categories

\begin{equation}
[\widehat{\Delta[1]},\mathbb{E}]\xrightarrow{\simeq}\mathrm{Fam}(\mathbb{E}^{*/})
\end{equation}

Due to this equivalence, we can slightly generalize Corollary \ref{cot} to obtain the following:
\end{rem}
\begin{cor}
\textit{Let $\mathbb{E}$ be a topos with terminal object $*$. Then $[\widehat{\Delta[1]},\mathbb{E}]$ is a locally connected topos. Additionally for any object $X\in[\widehat{\Delta[1]},\mathbb{E}]$, one has an isomorphism:}
\begin{equation}
\pi_{0}(X)\simeq\mathrm{Hom}_{\mathbb{E}}(*,X[0])\\
\end{equation}
\end{cor}
Corollary 6.6 indicates that the family construction may be of use in computations involving connected components of objects in locally connected topoi.\\
The next proposition gives a relationship between the coproduct and colimit completions of loci.
\begin{prop}\label{geo}
Let $\mathcal{C}$ be a locus. Then the Yoneda extension $(-)\otimes_{\mathcal{C}}\sigma:[\mathcal{C}^{op},\mathrm{Set}]\rightarrow\mathrm{Fam}(\mathcal{C})$ of the singleton embedding $\sigma:\mathcal{C}\hookrightarrow\mathrm{Fam}(\mathcal{C})$ is the inverse image component of a canonical geometric morphism $\mathrm{Fam}(\mathcal{C})\rightarrow[\mathcal{C}^{op},\mathrm{Set}]$. 
\end{prop}

\begin{proof}
By classical results, we can explicitly construct the right adjoint component $N$ by $N(\zeta)=\mathrm{Hom}_{\mathrm{Fam}(\mathcal{C})}(\sigma(-),\zeta):\mathcal{C}^{op}\rightarrow\mathrm{Set}$ for $\zeta\in\mathrm{Fam}(\mathcal{C})$. By [6, VII.9.1], In order to show that $(-)\otimes_{\mathcal{C}}\sigma$ is left-exact, it suffices to show that $\sigma$ is flat in the internal logic of $\mathcal{C}$. But by our assumptions $\mathcal{C}$ is locally presentable (and hence finitely complete) and $\mathrm{Fam}(\mathcal{C})$ is a topos, so $\sigma$ is internally flat if it is representably flat. Again by the finite completeness assumption on $\mathcal{C}$, $\sigma$ is representably flat precisely if it preserves finite limits, which holds by [2, Corollary 6.2.7]. Thus $(-)\otimes_{\mathcal{C}}\sigma$ is left-exact, so we obtain the desired geometric morphism $(-)\otimes_{\mathcal{C}}\sigma\dashv N:\mathrm{Fam}(\mathcal{C})\rightarrow[\mathcal{C}^{op},\mathrm{Set}]$. 
\end{proof}
\begin{defn}[Joyal]
An \textit{$\infty$-locus} is a locally presentable $\infty$-category $\mathcal{C}$ such that $\mathrm{Fam}(\mathcal{C})$ is an $\infty$-topos.
\end{defn}

\begin{exmp}
Important examples include the $\infty$-category $\mathrm{Grpd}_{\infty}^{*/}$ of pointed homotopy types and the $\infty$-category $\mathbf{Sp}$ of spectra. In analogy to the 1-categorical case, $\mathrm{Fam}(\mathrm{Grpd}_{\infty}^{*/})$ is an $\infty$-topos due to the equivalence $[\widehat{\Delta[1]},\mathrm{Grpd}_{\infty}]\simeq\mathrm{Fam}(\mathrm{Grpd}_{\infty}^{*/})$.
\end{exmp}
The following was conjectured by Joyal in \cite{joy} and proven by Hoyois in \cite{gal}. The proof in \textit{loc cit.} relies on the ``stable" Giraud Theorem.
\begin{thm}\label{stable}
Any locally presentable stable $\infty$-category is an $\infty$-locus.
\end{thm}

%We will provide a slight variant on Hoyois's proof of this %statement. We will use the following lemmata.

%\begin{lem}
%Let $X$ be a small $\infty$-category. Then the category $\mathbf{Sp}^{X}$ of spectra-valued presheaves on $X$ is a locus.
%\end{lem}
%\begin{proof}
%This is implied by the contents of [3] - in particular the observation that loci are closed under the formation of functor categories $\lambda\mapsto\lambda^{X}$. The claim then follows from the fact that $\mathrm{Fam}(\mathbf{Sp})$ is an $\infty$-topos [4].
%\end{proof}
%\begin{proof}[Alternative Proof of Theorem \ref{stable}]

%\end{proof}

%Let $\mathcal{P}\mathrm{r}\mathrm{Cat}^{\mathrm{st}}_{\infty}$ denote the $\infty$-category of presentable stable $\infty$-categories and exact functors and let $\mathfrak{Topos}_{\infty}$ denote the $\infty$-category of $\infty$-topoi and geometric morphisms. The above theorem gives evidence that the $\mathrm{Fam}$ construction provides an embedding $\mathrm{Fam}:\mathcal{P}\mathrm{r}\mathrm{Cat}^{\mathrm{st}}_{\infty}\hookrightarrow\mathfrak{Topos}_{\infty}$ of ``stable category theory" into higher topos theory. 
\section{Future Work}
In this paper, we have shown that the $\infty$-category $\mathrm{Fam}(\mathcal{C})$ has a natural homotopy theory which generalizes the homotopy theory of topological spaces. The following are some directions/questions for future work on this topic:\\

\begin{itemize}
    \item Generally, if $\mathcal{C}$ is an $(\infty,n)$-category for some $n\geq 1$, then the objects of $\mathrm{Fam}(\mathcal{C})$ are pairs $(X,F)$, where $X$ is an $(\infty,n-1)$ category and $F:X\rightarrow\mathcal{C}$ is a functor. This implies that an analog of $\Pi_{\infty}$ for parametrized families in arbitrary $(\infty,n)$-categories would output a directed space instead of a homotopy type (the ``fundamental $(\infty,n-1)$-category" instead of fundamental $\infty$-groupoid). There should be analogs of the contents of this paper in directed homotopy theory. However, there is a general lack of literature on the notions of $(\infty,n)$-topoi, etc., so this may be more difficult.\\
    
    \item  Can the analogy between the fundamental $\infty$-groupoid of a parametrized family and (for instance) the fundamental $\infty$-groupoids of objects in a locally $\infty$-connected higher topos be made more precise? More specifically, if $\mathrm{Fam}(\mathcal{C})$ is an $\infty$-topos, then does $\Pi_{\infty}$ fit into an adjoint triple resembling an essential geometric morphism?\\
\end{itemize}

\subsection{Acknowledgements}
We'd like to thank Nicholas Scoville for general discussions about this project and Marc Hoyois for clarifying our understanding about parametrized families in $\infty$-categories.

\end{document}